\documentclass{aumod}
\usepackage{amsmath}
\usepackage{amscd,amssymb,amsthm} %, amsmath are included by default
\usepackage{latexsym,mathrsfs,enumerate,scalefnt,ifthen}
\usepackage{stmaryrd}
\SetSymbolFont{stmry}{bold}{U}{stmry}{m}{n}
\usepackage{mathtools}
\usepackage[mathcal]{euscript}
\usepackage[colorlinks=true,urlcolor=black,linkcolor=black,citecolor=black]{hyperref}
\usepackage{scalefnt}
\usepackage{tikz}
\usepackage{color}
\usepackage[margin=1in]{geometry}
\usepackage{scrextend}
\usepackage{comment}

\numberwithin{equation}{section}
\theoremstyle{plain}
\newtheorem{theorem}{Theorem}[section]
\newtheorem{lemma}[theorem]{Lemma}

\theoremstyle{definition}

\newtheorem*{remarks}{Remarks}

\newcommand{\Psiba}{\ensuremath{\Psi_{\beta, \alpha}}}
\newcommand{\Deltaba}{\ensuremath{\Delta_{\beta, \alpha}}}
\newcommand{\Deltabar}{\ensuremath{\mathrel{\Delta_{\beta, \alpha}}}}
\usepackage[smaller]{acronym}
\usepackage{xspace}

\acrodef{lics}[LICS]{Logic in Computer Science}
\acrodef{sat}[SAT]{satisfiability}
\acrodef{nae}[NAE]{not-all-equal}
\acrodef{ctb}[CTB]{cube term blocker}
\acrodef{tct}[TCT]{tame congruence theory}
\acrodef{wnu}[WNU]{weak near-unanimity}
\acrodef{CSP}[CSP]{constraint satisfaction problem}
\acrodef{MAS}[MAS]{minimal absorbing subuniverse}
\acrodef{MA}[MA]{minimal absorbing}
\acrodef{cib}[CIB]{commutative idempotent binar}
\acrodef{sd}[SD]{semidistributive}
\acrodef{NP}[NP]{nondeterministic polynomial time}
\acrodef{P}[P]{polynomial time}
\acrodef{PeqNP}[P $ = $ NP]{P is NP}
\acrodef{PneqNP}[P $ \neq $ NP]{P is not NP}

\usepackage{macros}

\begin{document}

%% \title[Computing the Commutator Efficiently]{Computing the commutator efficiently}
\title[Commutator as a Least Fixed Point]{%
  The Commutator as Least Fixed Point of a Closure Operator}
\date{7 March 2017}
\author[William DeMeo]{William DeMeo\\
7 March 2017}
\address{University of Hawaii, Honolulu 96822}
\address{e-mail: \href{mailto:williamdemeo@gmail.com}{williamdemeo@gmail.com}}

\begin{abstract}
  We present a description of the (non-modular) commutator, inspired by that of
  Kearnes in~\cite[p.~930]{MR1358491}, that provides a simple recipe for computing
  the commutator. 
\end{abstract}

\maketitle

\newcommand\mytimes{\ensuremath{\ast}}
\section{Preliminaries}
If $A$ and $B$ are sets and  $\alpha \subseteq A\times A$ and $\beta \subseteq B\times B$
are binary relations on $A$ and $B$, respectively, then 
we define
the \emph{pairwise product} of $\alpha$ and $\beta$ by
\begin{equation}
\label{eq:pair-product}
\alpha \mytimes \beta = \{((a, b), (a', b')) 
\in (A\times B)^2 \mid a\mathrel{\alpha} a'\, \text{ and } \,  b\mathrel{\beta} b'\},
\end{equation}
and we let $\alpha \times \beta$ denote the usual Cartesian product of sets; that is,
%% of the sets $\alpha$ and $\beta$, that is,
\begin{equation}
\label{eq:set-product}
\alpha \times \beta = \{((a, a'), (b, b')) 
\in A^2\times B^2 \mid a\mathrel{\alpha} a' \, \text{ and } \, b\mathrel{\beta} b'\}.
\end{equation}
The equivalence class of $\alpha \mytimes \beta$ containing the pair
$(a, b)$ is denoted and defined by % \in A\times B$ is
\[(a,b)/(\alpha \mytimes \beta) = a/\alpha \times b/\beta= 
    \{(a', b') \in A\times B \mid a\mathrel{\alpha} a' \, \text{ and } \,  b\mathrel{\beta} b'\},
    \]
the Cartesian product of the sets $a/\alpha$ and $b/\beta$.
The set of all equivalence classes of $\alpha \mytimes \beta$ is also a Cartesian product, namely,
$(A\times B)/(\alpha \mytimes \beta) =
A/\alpha \times B/\beta  = \{(a, b)/(\alpha \mytimes \beta) \mid a\in A \, \text{ and } \, b \in B\}$.

For an algebra $\bA$ with congruence relations $\alpha$, $\beta\in \Con\bA$,
let $\bbeta$ denote the subalgebra of $\bA\times \bA$ with universe 
$\beta$, and let $0_A$ denote the least equivalence relation on $A$.
Thus, $0_A = \{(a,a) \mid a\in A\} \leq \beta$.
Denote by $D_\alpha$ the following subset of $\beta \times \beta$:
\begin{equation}
  \label{eq:9009}
D_\alpha =(\alpha \mytimes \alpha) \cap (0_A \times 0_A)
= \{((a,a), (b,b)) \in (0_A \times 0_A) \mid a\alphar b\}.
\end{equation}
Let $\Delta_{\beta, \alpha} = \Cg^{\bbeta}(D_\alpha)$ denote the congruence relation
of $\bbeta$ generated by $D_\alpha$.
The condition $\CC{\alpha}{\beta}{\gamma}$
holds iff for all $a \alphar b$, for all $u_i \betar v_i$ ($1\leq i\leq n$), and for all 
$t\in \Pol_{n+1}(\bA)$, we have
$t(a,\bu) \mathrel{\gamma} t(a, \bv)$
iff $t(b,\bu) \mathrel{\gamma} t(b, \bv)$.
There are a number of different ways to define a commutator.
See, for example,~\cite{MR0432511,MR541622,MR590312,MR1145556,MR1663558,MR1257643}.
The present note concerns the commutator $[\alpha, \beta]$ defined to be the least
congruence $\gamma$ such that $\CC{\alpha}{\beta}{\gamma}$ holds.
\begin{comment}
Occasionally it is more convenient to write such an equivalence as a (two-way) derivation tree,
as follows:
\[
\infer=[\CC{\alpha}{\beta}{\gamma}]{\Gamma \vdash t(b,\bu) \mathrel{\gamma} t(b, \bv)}{\Gamma \vdash t(a,\bu) \mathrel{\gamma} t(a, \bv)}\]
where $\Gamma$ is a context containing
$a \alphar b$, $u_i \betar v_i$ ($1\leq i\leq n$), and 
$t\in \Pol_{n+1}(\bA)$.
\end{comment}

\section{Alternate Description of the Commutator}
\label{sec:altern-descr}

We now %% In~\cite{com-fix-poi} we gave 
describe an alternate way to express the commutator---specifically,
it is the least fixed point of a certain closure operator.
This description was inspired by the one that is mentioned in passing by
Keith Kearnes in~\cite[p.~930]{MR1358491}.  Our objective here is
to prove that the description we present is correct (i.e., describes the commutator)
and to show that it leads to a simple, efficient procedure for computing the commutator.

Let $\Tol(A)$ denote the collection of all tolerances (reflexive symmetric relations)
on the set $A$,\footnote{Actually, a
  \emph{tolerance} of an algebra $\bA = \<A, \dots\>$
  is a reflexive symmetric subalgebra of $\bA \times \bA$.
  Therefore, the set of all tolerances of $\bA$ forms an
  algebraic (hence complete) lattice.
  If we drop the operations and consider only the set $A$, then a tolerance relation on $A$
  is simply a reflexive symmetric binary relation.
}
and let
$\Psi_{\beta, \alpha} \colon \Tol(A) \to \Tol(A)$ be the function defined
for each $T \in  \Tol(A)$ follows:
\begin{equation}
  \label{eq:7}
  \Psi_{\beta, \alpha}(T)
  = \{ (x,y) \in A\times A \mid
  (\exists\, (a,b) \in T)\,
 (a,b) \mathrel{\Delta_{\beta, \alpha}} (x,y)\},
\end{equation}
where
$\Delta_{\beta, \alpha} = \Cg^{\bbeta}(D_\alpha)$ and
$D_\alpha =(\alpha \mytimes \alpha) \cap (0_A \times 0_A)$
(as in~(\ref{eq:9009})).

\pagebreak[2]
\begin{remarks}\
  \begin{enumerate}
\item
  It's easy to see that $\Psiba (T)$ is reflexive and symmetric
  whenever $T$ has these properties; similarly, $\Psiba (T)$ is compatible with the
  operations of $\bA$ whenever $T$ is. In other words $\Psiba$ maps tolerances of
  $A$ ($\bA$, resp.) to tolerances of $A$ ($\bA$, resp.).
  \item 
  Since $\Psiba$ is clearly a monotone increasing function on the complete
  lattice $\Tol(A)$, it is guaranteed to have a least fixed
  point---that is, there is a point $\tau\in \Tol(A)$ such that $\Psiba(\tau) = \tau$
  and $\tau \leq T$, for every $T \in \Tol(A)$
  satisfying $\Psiba(T) = T$.
\item
  Here are two ways the least fixed point of $\Psiba$ could be computed:
  \begin{equation}
    \label{eq:4}
  \tau = \Meet \{ T \in \Tol(A) \mid \Psiba(T) \leq T\}
  \quad \text{ and } \quad
     \tau = \Join_{k\geq 0} \Psiba^{k}(0_A).
  \end{equation}
  In Lemma~\ref{lem:fixed-point-comm} we will show that the least
  fixed point of $\Psiba$ is, in fact, the commutator,
  $\tau = [\alpha, \beta]$, so either
  expression in~(\ref{eq:4}) could potentially be used to compute it.
  %% For example,
  %% an algorithm might be based on the following formula:
  %% \begin{equation}
  %%   \label{eq:5}
  %%         [\alpha, \beta] = \Join_{k\geq 0} \Psiba^{k}(0_A).
  %% \end{equation}
  However, Lemma~\ref{lem:fixed-point-comm} also shows that 
  $\Psiba$ is a closure operator; in particular, it is idempotent. Therefore,
  $\Psiba^{k}(0_A) = \Psiba(0_A)$ for all $k$, so we have 
  the following simple description of the commutator:
  \begin{align*}
    %% \label{eq:55}
          [\alpha, \beta] =
          \Psiba(0_A)
          &= \{ (x,y) \in A\times A \mid
          (\exists\, (a,b) \in 0_A)\, (a,b) \mathrel{\Delta_{\beta, \alpha}} (x,y)\}\\
          &= \{ (x,y) \in A\times A \mid
          (\exists a \in A)\, (a,a) \mathrel{\Delta_{\beta, \alpha}} (x,y)\}.
  \end{align*}

  %(See, for example,~\cite{MR3012378}.)

  %% Recall, if $f$ is a monotone increasing function defined on a
  %%   complete poset $\<P, \leq\>$, then the least fixed point of $f$
  %%   is $\Meet \{ p\in P \mid f p \leq p\}$. %(See, for example,~\cite{MR3012378}.)
  %%   Thus,
  %%   Lemma~\ref{lem:fixed-point-comm}~(\ref{item:2}) asserts that
  %%   \begin{equation}
  %%     \label{eq:2}
  %%           [\alpha, \beta] =\Meet \{ B \subseteq \beta \mid \Psiba(B) \subseteq B\}.
  %%   \end{equation}
  \end{enumerate}
\end{remarks}

\subsection{Fixed Point Lemma}
\begin{lemma}
  \label{lem:fixed-point-comm}
  If $\alpha$, $\beta\in \Con(\bA)$ and 
  if $\Psi_{\beta, \alpha}$ is defined by~(\ref{eq:7}), then 
  \begin{enumerate}[(i)]
  \item \label{item:1} $\Psiba$ is a closure operator on $\Tol(A)$;
  \item \label{item:2} $[\alpha, \beta]$ is the least fixed point of $\Psiba$.
  \end{enumerate}
\end{lemma}
\begin{proof}\
  \begin{enumerate}[(i)]
  \item 
  %% In fact, $\Psi_{\beta, \alpha}$ is a closure operator on all
  %% of $\Tol(A)$ as we 
    To prove (i) we verify that
    $\Psi_{\beta, \alpha}$ has the three properties that define a closure
    operator---namely for all $T$, $T' \in \Tol(A)$,
  \begin{enumerate}[(c.1)]
  \item \label{item:c1} $T  \leq \Psiba(T )$;     
  \item \label{item:c2} $T  \leq T'  \Rightarrow \Psiba(T) \leq \Psiba(T')$;    
  \item \label{item:c3} $\Psiba(\Psiba(T))  = \Psiba(T)$. 
  \end{enumerate}

  \vskip2mm

  \noindent {\it Proof of (c.1):} $(a,b) \in T $
  implies $(a,b) \in \Psiba(T )$ because $(a,b)\mathrel{\Delta_{\beta, \alpha}} (a,b)$.\\[4pt]
  %% this proves~(c.\ref{item:c1}).
  \noindent {\it Proof of (c.2):} $(x,y) \in \Psiba(T )$ iff there exists
  $(a,b) \in T  \leq T'$ such that
  $(a,b) \mathrel{\Delta_{\beta, \alpha}} (x,y)$; this and $(a,b) \in T'$ implies
  $(x,y) \in \Psiba(T')$.\\[4pt]
  \noindent {\it Proof of (c.3):} $(x,y) \in \Psiba(\Psiba(T))$ if and only if
  there exists $(a,b) \in \Psiba(T)$ such that
  $(a,b) \Deltabar (x,y)$, and $(a,b) \in \Psiba(T)$ is in turn equivalent to 
  the existence of $(c,d) \in T $ such that
  $(c,d) \Deltabar (a,b)$. By transitivity of $\Deltaba$, we have that
  $(c,d) \Deltabar (a,b) \Deltabar (x,y)$ implies
  $(c,d) \Deltabar (x,y)$, proving that there exists $(c,d) \in T $ such that
  $(c,d) \Deltabar (x,y)$; equivalently, $(x,y) \in T $.

  \medskip

\item
  %% \noindent (ii) 
  As remarked above, from part (i) follows 
  $\Psiba^{k}(0_A) = \Psiba(0_A)$ for all $k$, so the least fixed point of
  $\Psiba$ that appears in the formula on the right in~(\ref{eq:4}) reduces
  to $\tau = \Psiba(0_A)$.  Therefore, to complete the proof it suffices to show
  $[\alpha, \beta] = \Psiba(0_A)$.
  %% An alternative direct proof using
  %% \malcev's congruence generation theorem appears in the appendix Section
  %% below.
  %% [\alpha, \beta] = \Join_{k\geq 0} \Psiba^{k}(0_A).

  We first prove $[\alpha, \beta]\leq \Psiba(0_A)$.
  Since $[\alpha, \beta]$ is the least congruence $\gamma$
  satisfying $\CC{\alpha}{\beta}{\gamma}$, it suffices to prove
    $\CC{\alpha}{\beta}{\Psiba(0_A)}$ holds.
    Suppose $a \alphar a'$ and $b_i \betar b_i'$ %% ($1\leq i \leq k$)
    and $t^{\bA} \in \Pol_{k+1}(\bA)$ satisfy
    $t^{\bA}(a, \bb) \mathrel{\Psiba(0_A)} t^{\bA}(a, \bb')$,
    where $\bb = (b_1, \dots, b_k)$ and $\bb' = (b_1', \dots, b_k')$.
    We must show $t(a', \bb) \mathrel{\Psiba(0_A)} t(a', \bb')$.  
    By definition of $\Psiba$,
    the antecedent $t^{\bA}(a, \bb) \mathrel{\Psiba(0_A)} t^{\bA}(a, \bb')$ is equivalent to    
    the existence of $c \in A$ such that $(c,c) \Deltabar (t^{\bA}(a, \bb), t^{\bA}(a, \bb'))$.
    Now
    \[
    (t^{\bA}(a, \bb), t^{\bA}(a, \bb')) = t^{\bbeta}((a,a),(b_1, b_1'), \dots,(b_k, b_k')),
    \]
    and since $a \alphar a'$, we have
    \[
    t^{\bbeta}((a,a),(b_1, b_1'), \dots,(b_k, b_k'))
    \Deltabar
    t^{\bbeta}((a',a'),(b_1, b_1'), \dots,(b_k, b_k')).
    \]
    The latter is equal to $(t^{\bA}(a', \bb), t^{\bA}(a', \bb'))$, and  it follows
    by transitivity of $\Deltaba$ that
    $(c,c) \Deltabar (t^{\bA}(a', \bb), t^{\bA}(a', \bb'))$.
    Therefore, $t(a', \bb) \mathrel{\Psiba(0_A)} t(a', \bb')$, as desired.

  We now prove $\Psiba(0_A)\leq   [\alpha, \beta]$.
  %% \begin{equation}
  %%   \label{eq:8}
  %% \Join_{k\geq 0} \Psiba^{k}(0_A)\leq   [\alpha, \beta].
  %% \end{equation}
  If $(x,y)\in \Psiba(0_A)$ then there exists $a \in A$ such that 
  \begin{equation}
    \label{eq:1100}
    (a,a) \mathrel{\Delta_{\beta, \alpha}} (x,y).
  \end{equation}
  From the definition of $\Delta_{\beta, \alpha}$ and 
  \malcev's congruence generation theorem,~(\ref{eq:1100})
  holds if and only if for there exist
  $(z_i, z_i') \in \beta$ ($0\leq i \leq n+1$), and $(u_i, v_i) \in \alpha$,
  $f_i \in \Pol_1(\bbeta)$ ($0\leq i \leq n$), such that
  $(a, a) = (z_0,z_0')$ and $(x, y)=(z_{n+1},z'_{n+1})$ hold, and so do the
  following equations of sets: 
  \begin{align}
    \label{eq:001}
    \{(a, a),(z_1,z_1')\} &= \{f_0(u_0,u_0), f_0(v_0,v_0)\},\\
    \label{eq:011}
    \{(z_1,z_1'),(z_2,z_2')\} &= \{f_1(u_1,u_1), f_1(v_1,v_1)\},\\
    \nonumber
    &\; \; \vdots\\
    %% \label{eq:n-1}
    \nonumber
    %% \{(z_{n-1},z_{n-1}'),(x, y)\} &= \{f_{n-1}(u_{n-1},u_{n-1}), f_{n-1}(v_{n-1},v_{n-1})\}.
    \{(z_{n},z_{n}'),(x, y)\} &= \{f_{n}(u_{n},u_{n}), f_{n}(v_{n},v_{n})\}.
  \end{align}
  Now $f_i \in \Pol_1(\bbeta)$ for all $i$, so
  \newcommand\gA{\ensuremath{g^{\bA}}}%
  \[
  f_i(c, c') = g_i^{\bbeta}((c, c'), (b_1, b_1'), \dots, (b_k, b_k') )
  = (\gA_i(c, \bb), \gA_i(c', \bb')),%
  \]
  \renewcommand\gA{\ensuremath{g}}%
  for some $k$, some $(k+1)$-ary term $\gA_i$, and some constants
  $\bb = (b_1, \dots, b_k)$ and $\bb' = (b_1', \dots, b_k')$ satisfying
  $b_i \betar b_i'$ ($1\leq i\leq k$). 
  By~(\ref{eq:001}), either
  \[
  (a, a) = \bigl(\gA_0(u_0, \bb), \gA_0(u_0, \bb')\bigr)
  \quad \text{ and } \quad 
  (z_1,z_1')= \bigl(\gA_0(v_0, \bb), \gA_0(v_0, \bb')\bigr),
  \]
  or vice-versa. %Of course $(a, a) \in \comm{\alpha}{\beta}$,
  We assumed $u_0 \alphar v_0$ and $b_i \betar b_i'$ ($1\leq i\leq k$),
  so the $\alpha,\beta$-term condition entails
  $\gA_0(u_0, \ba) \commr{\alpha}{\beta} \gA_0(u_0, \ba')$
  iff 
  $\gA_0(v_0, \ba) \commr{\alpha}{\beta} \gA_0(v_0, \ba')$.
  %% \[
  %%   \gA_0(u_0, \bb) \commr{\alpha}{\beta} \gA_0(u_0, \bb')
  %%   \quad \Longleftrightarrow \quad 
  %%   \gA_0(v_0, \bb) \commr{\alpha}{\beta} \gA_0(v_0, \bb').
  %%   \]
  From this and~(\ref{eq:001}) we deduce that 
  $(a,a)\in [\alpha, \beta]$ iff $(z_1,z_1')\in [\alpha, \beta]$.
  Similarly~(\ref{eq:011}) and $u_1 \alphar v_1$ imply
  $(z_1,z_1')\in [\alpha, \beta]$ iff
  $(z_2,z_2')\in [\alpha, \beta]$.  Inductively, and by transitivity of
  $[\alpha, \beta]$, we conclude $(a,a)\in [\alpha, \beta]$ iff
  $(x,y)\in [\alpha, \beta]$.
  Since $(a,a)\in [\alpha, \beta]$, we have $(x,y)\in [\alpha, \beta]$, as desired.

  \end{enumerate}
\end{proof}

\section{Computing the Commutator}
As a consequence of the description of the commutator given in the last section,
we now have the following simple method for computing it.

\smallskip

\noindent {\bf Input} \hskip2mm A finite algebra, $\bA = \<A, \dots\>$, and two congruence relations $\alpha$, $\beta \in \Con \bA$.

\smallskip
\noindent {\bf Procedure}
\begin{itemize}
\item {\bf Step 1} \hskip2mm Compute the congruence relation
  $\Deltaba = \Cg^{\bbeta}\bigl\{((a,a), (b,b)) \mid a \alphar b \bigr\}$.
\item {\bf Step 2} \hskip2mm Compute the commutator
  %% \begin{align*}
  %% [\alpha, \beta] &= \Psiba(0_A)
  %% = \{(x,y) \in A\times A \mid \bigl(\exists a \in A\bigr) \, (a,a) \Deltabar (x,y)\}\\
  %% &= \bigcup_{a\in A} (a,a)/\Deltaba
  %% \end{align*}
  \[[\alpha, \beta] 
  = \bigl\{(x,y) \in A\times A \mid (\exists a \in A) \, (a,a) \Deltabar (x,y)\bigr\} 
  =\bigcup_{a\in A} (a,a)/\Deltaba
    \]
\end{itemize}
Note that $\Deltaba$ is a subalgebra of $\bA^2 \times \bA^2$ and such a congruence
can be computed in polynomial-time in the size of $\bA$. (See~\cite{MR2470585}.)

\bibliographystyle{alphaurl}
\bibliography{refs}

\begin{thebibliography}{Gum80}

\bibitem[DG92]{MR1145556}
A.~Day and H.~P. Gumm.
\newblock Some characterizations of the commutator.
\newblock {\em Algebra Universalis}, 29(1):61--78, 1992.
\newblock URL: \url{http://dx.doi.org/10.1007/BF01190756}, \href
  {http://dx.doi.org/10.1007/BF01190756} {\path{doi:10.1007/BF01190756}}.

\bibitem[Fre08]{MR2470585}
Ralph Freese.
\newblock Computing congruences efficiently.
\newblock {\em Algebra Universalis}, 59(3-4):337--343, 2008.
\newblock URL: \url{http://dx.doi.org/10.1007/s00012-008-2073-1}, \href
  {http://dx.doi.org/10.1007/s00012-008-2073-1}
  {\path{doi:10.1007/s00012-008-2073-1}}.

\bibitem[Gum80]{MR590312}
H.-Peter Gumm.
\newblock An easy way to the commutator in modular varieties.
\newblock {\em Arch. Math. (Basel)}, 34(3):220--228, 1980.
\newblock URL: \url{http://dx.doi.org/10.1007/BF01224955}, \href
  {http://dx.doi.org/10.1007/BF01224955} {\path{doi:10.1007/BF01224955}}.

\bibitem[HH79]{MR541622}
Joachim Hagemann and Christian Herrmann.
\newblock A concrete ideal multiplication for algebraic systems and its
  relation to congruence distributivity.
\newblock {\em Arch. Math. (Basel)}, 32(3):234--245, 1979.
\newblock URL: \url{http://dx.doi.org/10.1007/BF01238496}, \href
  {http://dx.doi.org/10.1007/BF01238496} {\path{doi:10.1007/BF01238496}}.

\bibitem[Kea95]{MR1358491}
Keith~A. Kearnes.
\newblock Varieties with a difference term.
\newblock {\em J. Algebra}, 177(3):926--960, 1995.
\newblock URL: \url{http://dx.doi.org/10.1006/jabr.1995.1334}, \href
  {http://dx.doi.org/10.1006/jabr.1995.1334}
  {\path{doi:10.1006/jabr.1995.1334}}.

\bibitem[KS98]{MR1663558}
Keith~A. Kearnes and {\'A}gnes Szendrei.
\newblock The relationship between two commutators.
\newblock {\em Internat. J. Algebra Comput.}, 8(4):497--531, 1998.
\newblock URL: \url{http://dx.doi.org/10.1142/S0218196798000247}, \href
  {http://dx.doi.org/10.1142/S0218196798000247}
  {\path{doi:10.1142/S0218196798000247}}.

\bibitem[Lip94]{MR1257643}
Paolo Lipparini.
\newblock Commutator theory without join-distributivity.
\newblock {\em Trans. Amer. Math. Soc.}, 346(1):177--202, 1994.
\newblock URL: \url{http://dx.doi.org/10.2307/2154948}, \href
  {http://dx.doi.org/10.2307/2154948} {\path{doi:10.2307/2154948}}.

\bibitem[Smi76]{MR0432511}
Jonathan D.~H. Smith.
\newblock {\em Mal$\prime$cev varieties}.
\newblock Lecture Notes in Mathematics, Vol. 554. Springer-Verlag, Berlin-New
  York, 1976.

\end{thebibliography}

\end{document}